\documentclass[10pt, reqno]{amsart}
\usepackage{amsmath}
\usepackage{amscd}
\usepackage{amsthm}
\usepackage{amssymb} \usepackage{latexsym}
\usepackage{eufrak}
\usepackage{euscript}
\usepackage{epsfig}
\usepackage{graphics}
\usepackage{array}
\usepackage{enumerate}
\usepackage{dsfont}
\usepackage{color}
\usepackage{wasysym}

\newcommand{\bel}[1]{\begin{equation}\label{#1}}

\newcommand{\be}{\begin{equation}}

\newcommand{\ba}{\begin{eqnarray}}
\newcommand{\ea}{\end{eqnarray}}

\newcommand{\qe}{\end{equation}}
\newcommand{\R}{{\mathbb R}}
\newcommand{\N}{{\mathbb N}}
\newcommand{\Z}{{\mathbb Z}}

\newcommand{\supp}{{\mathrm{supp}}}

\theoremstyle{theorem}
\newtheorem{thm}{Theorem}[section]

\theoremstyle{example}
\newtheorem{example}{Example}[section]
\theoremstyle{corollary}
\newtheorem{coro}{Corollary}[section]
\theoremstyle{lemma}
\newtheorem{lemma}{Lemma}[section]
\theoremstyle{definition}

\theoremstyle{proof}

\theoremstyle{remark}
\newtheorem{rem}{Remark}[section]

\begin{document}

\title{$L^q$ Harmonic Functions on Graphs}

\author{Bobo Hua}
\email{bobohua@mis.mpg.de}
\address{Max Planck Institute for Mathematics in the Sciences\\
04103 Leipzig, Germany.}
\author{J\"urgen Jost}
 \email{jost@mis.mpg.de}
\address{Max Planck Institute for Mathematics in the Sciences\\
04103 Leipzig, Germany.}
\address{Department of Mathematics and Computer Science \\University of
Leipzig \\04109 Leipzig, Germany }

\thanks{The research leading to these results has received funding from the
European Research Council under the European Union's Seventh
Framework Programme (FP7/2007-2013) / ERC grant agreement n$^\circ$
267087.}

\begin{abstract}We prove an analogue of   Yau's Caccioppoli-type
  inequality for nonnegative subharmonic functions on graphs. We then
  obtain a Liouville theorem for harmonic or non-negative subharmonic
  functions of class $L^q$, $1\le q <\infty$, on any graph, and a
  quantitative version for $q>1$. Also, we provide counterexamples for
  Liouville theorems for $0<q<1$.
\end{abstract}
\maketitle

\section{Introduction}
In 1976, Yau \cite{Yau76} proved an $L^q$ ($1<q<\infty$) Liouville
theorem for harmonic functions on complete Riemannian manifolds. Yau's
theorem says that  there doesn't exist any nonconstant $L^q$ ($1<q<\infty$)
harmonic functions on any complete Riemannian manifold $M$. This
result is quite remarkable since it does not require any assumption
besides the -- obviously necessary -- completeness on the underlying manifold.

Karp \cite{Karp82} then found a quantitative version of Yau's $L^q$
Liouville theorem. Let $f$ be a nonconstant nonnegative subharmonic
function on $M.$ Then
\begin{equation}\label{eqi1}\liminf_{R\to\infty}\frac{1}{R^2}\int_{B_R(p)}f^qd\mathrm{vol}=\infty,\end{equation}
where $B_R(p)$ is the geodesic ball centered at $p$ of radius $R$
and $1<q<\infty$. After that, Li-Schoen \cite{LiSchoen84} proved an
$L^q$ mean value inequality for subharmonic functions on manifolds
with proper curvature conditions, which implies the $L^q$ Liouville
theorem for such manifolds.

In 1997, Rigoli-Salvatori-Vignati \cite{RigoliSalvatoriVignati97}
generalized Karp's version of the $L^q$ Liouville theorem to the graph
setting. Under the assumption of uniformly bounded degree for the
graph, they proved the analogue of \eqref{eqi1} (with sums in place of
integrals) for the case $q\geq 2.$ The case $1<q<2$ was left open. In this paper, we use an idea of Yau, the
$L^q$ Caccioppoli-type inequality (see Theorem \ref{Caccioppolli}),
to prove a theorem that resolves that case.

In order to  introduce our setting for graphs, let $G=(V,E)$ be an
infinite, connected, locally finite, weighted graph. Each edge $e\in
E$ carries a positive weight $\mu_e$, and for each vertex $x\in V$
this then induces the positive weight $\mu_x=\sum_{y\sim
x}\mu_{xy},$ called the vertex degree of $x$, where $y\sim x$ means
that they are neighbors. We denote by $B_R(p)$ the closed ball
centered at $p$ of radius $R$ in $G$ where the distance between two
vertices is given by the minimal number of edges to be traversed
when going from one to the other.

The (normalized) Laplace operator is defined as
$$\Delta f(x)=\frac{1}{\mu_x}\sum_{y\sim x}\mu_{xy}\left(f(y)-f(x)\right).$$ A function
$f$ is called harmonic (subharmonic) if $\Delta f=0 (\geq 0).$

We can now formulate our main result.
\begin{thm}\label{MT}
Let $f$ be a nonnegative subharmonic function on the weighted graph
$G.$ Then, either $f$ is constant or, for any $q\in (1,\infty),$
\begin{equation}\label{meq1}
\liminf_{R\to \infty}
\frac{1}{R^2}\sum_{B_R(p)}f^{q}(x)\mu_x=\infty.
\end{equation}
\end{thm}
Note that in our theorem, in contrast to
\cite{RigoliSalvatoriVignati97}, we do not need to assume any
uniform upper and lower bounds of the vertex degree of the graph.
This is consistent with the intuition that nonconstant harmonic
functions on infinite graphs of finite volume grow extremely fast.
We will give an example to show that \eqref{meq1} is false for the
case $q=1,$ see Remark \ref{remq}.

It is well-known that for graphs with $\mu_x\geq \mu_0>0$ for all
$x\in V,$ the $L^q$ ($0<q<\infty$) Liouville theorem is an easy
consequence of the maximum principle for subharmonic functions, see
Theorem \ref{Mxp}. Thus, the only interesting case of the $L^q$
Liouville theorem concerns general weighted graphs. In the case of
Riemannian manifolds, soon after Yau's $L^q$ ($1<q<\infty$)
Liouville theorem, Chung \cite{Chung83} provided an example to show
that there is no $L^1$ Liouville theorem. Counterexamples for an
$L^q$ ($0<q\leq 1$) Liouville theorem were then given in Li-Schoen
\cite{LiSchoen84}. That is why Li-Schoen \cite{LiSchoen84} proved
the $L^q$ ($0<q<\infty$) Liouville theorem under an additional
curvature assumptions. Surprisingly, we can adopt an idea of Li
\cite{Li84,Li12} to prove the $L^q$ ($1\leq q<\infty$) Liouville
theorem, even for the borderline case $q=1$. Although the curvature
assumptions are necessary in the Riemannian case by Li \cite{Li84},
on graphs we don't need any such curvature-like assumptions. There
are other generalizations of the $L^q$ Liouville theorem, for
instance, by Sturm \cite{Sturm94} to strongly local regular
Dirichlet forms (not including graphs, $q\neq1$) and Masamune
\cite{Masamune09} to graphs with different weights ($2\leq q\in\N$).
We can show
\begin{thm}\label{Lqharm}
For any graph $G,$ there doesn't exist any nonconstant $L^q$
harmonic (nonnegative subharmonic) function for $q\in
[1,\infty)$.\end{thm} We give examples of a large class of graphs
with infinite volume that don't satisfy the $L^q$ Liouville
theorem for any $0<q<1,$ see Example \ref{expp1}.

As an application, we study the $L^q$ Liouville theorem for higher
order elliptic operators on graphs where the maximum principle is
no longer available.

The organization of the paper is as follows: The basic facts on
graphs are collected in Sect.~\ref{sec2}, the next section contains
the proof of Yau's Caccioppoli-type inequality and the solution to
the problem of Rigoli-Salvatori-Vignati
\cite{RigoliSalvatoriVignati97}, and the last section is devoted to
the
$L^1$ Liouville theorem on graphs and an application to higher order
operators.

\section{Preliminaries}\label{sec2}
Let $G=(V,E)$ be an infinite, connected, locally finite, weighted
graph (see e.g. \cite{Chung97,Grigoryan09} for definitions).
$G=(V,E)$ is a weighted graph with positive and symmetric edge
weights $\mu_{xy}>0$ for any $xy\in E.$  For convenience, we extend
the edge weight to $V\times V$ by $\mu_{xy}=0$ for $xy\not \in E.$
The graph $G$ may have self-loops, i.e. $xx\in E$ (or $\mu_{xx}>0$)
but w.l.o.g. we exclude multiple edges because they can be  implicitly
encoded in the edge weights
$\mu_{xy}.$ Define the measure on $V$ as $\mu_x=\sum_y\mu_{xy}$ for
$x\in V.$ Let us denote by $\mu(G):=\sum_x\mu_x$ the total volume of
the graph $G.$ There are many interesting graphs of finite volume.
The Laplace operator on $G$ is defined as
$$\Delta f(x)=\sum_{y}\frac{\mu_{xy}}{\mu_x}\left(f(y)-f(x)\right).$$ The transition operator associated
to the random walk on the graph is defined as
$Pf(x)=\sum_{y}P(x,y)f(y)$ where $P(x,y)=\frac{\mu_{xy}}{\mu_x}.$
Obiviously, $\Delta=P-I$ where $I$ is the identity operator.

There is a natural (combinatorial) distance function $d$ on the
graph, simply counting the minimal number of edges separating two vertices. We denote by $B_R(p):=\{x\in V: d(x,p)\leq R\}$ the closed
ball centered at $p$ of radius $R.$ For any subset $\Omega\subset
G,$ we denote by $d(x,\Omega):=\min\{d(x,y):y\in \Omega\}$ the
distance to $\Omega,$ by $\partial \Omega:=\{y\in G:d(x,\Omega)=1\}$
the boundary of $\Omega.$ A function $f:\Omega\cup\partial
\Omega\rightarrow \R$ is called \emph{harmonic (subharmonic,
superharmonic)} on $\Omega$ if $\Delta f(x)=0$ ($\geq 0,\leq 0$) for
all $x\in\Omega.$ For any function $f$ on $G$ we denote the $L^q$
norm of $f$ by $\|f\|_q:=(\sum_x |f(x)|^q\mu_x)^{1/q},$
$q\in(0,\infty).$

For our difference operators, we need orientations. We choose an
orientation for the edge set $E,$ that is, $e=xy$ means that the
edge $e$ starts at $x$ and ends at $y.$ Let us denote by $\R^V:=\{f:
V\rightarrow\R\}$ (resp. $\R^E$) the set of all functions on $V$
(resp. on $E$), by $C_c(V)$ (resp. $C_c(E)$) the space of compact
supported functions defined on $V$ (resp. on $E$). We define inner
products on $C_c(V)$
and $C_c(E)$ as \begin{eqnarray*} \langle f,g\rangle&=&\sum_{x\in V} f(x)g(x)\mu_x,\\
\langle u,v\rangle&=&\sum_{e\in E} u(e)v(e)\mu_e,
\end{eqnarray*} where $f,g\in C_c(V)$ and $u,v\in C_c(E).$ For any
$f\in \R^V,$ the pointwise gradient $\nabla f\in \R^E$ is defined as
$\nabla f(e)=\nabla_ef:= f(y)-f(x)$ for all $e=xy\in E$. A very
useful formula reads as, for any $e=xy,$
$$\nabla_e(fg)=f(x)\nabla_eg+\nabla_ef g(y).$$ In addition, we have
Green's formula (see e.g. \cite{Grigoryan09}), for $f\in \R^V$ and
$g\in C_c(V),$
\begin{equation*}
\langle\Delta f, g\rangle =-\langle\nabla f, \nabla g\rangle.
\end{equation*}
In the following, we mean by $e\subset A$ for some subset
$A\subset V$ that both vertices of the edge $e$ are contained in $A$.

The graph is called \emph{non-degenerate} if $\mu_x\geq \mu_0>0$ for
all $x\in V.$ As a well-known result, we will see that there are no nontrivial $L^q$ nonnegative subharmonic functions on
non-degenerate graphs for any $q\in (0,\infty).$ This follows from
the maximum principle for subharmonic functions and the uniform lower
bound of the measure of a non-degenerate graph.
\begin{lemma}[Maximum principle]
Let $\Omega$ be a finite connected subset of $G$ and $f$ be
subharmonic on $\Omega.$ Then
\begin{equation}
\max_{\Omega}{f}\leq\max_{\partial\Omega}f,
\end{equation} where the equality holds iff $f$ is constant on $\Omega\cup\partial \Omega$.
\end{lemma}
\begin{proof}
This follows from the definition of subharmonic functions and the
connectedness of $\Omega$ (see \cite{Grigoryan09}).
\end{proof}

\begin{thm}\label{Mxp}
Let $G$ be a non-degenerate graph. Then there doesn't exist any
nontrivial $L^q$ nonnegative subharmonic functions on $G$,
$q\in(0,\infty)$.
\end{thm}
\begin{proof}
Let $f$ be an $L^q$ nonnegative subharmonic function. Since $G$ is a
non-degenerate graph and $f\in L^q$, we have $f(x)\to 0$ as
$x\to\infty.$ By the maximum principle, this yields that $|f|\leq
\epsilon$ for any $\epsilon>0.$ That is, $f\equiv0.$
\end{proof}

If $G$ is not non-degenerate, we cannot apply the maximum principle
as above. In fact, there are $L^q$ harmonic functions on some
 graphs which are unbounded at infinity.

\section{$L^q$ subharmonic and harmonic functions}
In this section, we generalize Yau's $L^q$ Caccioppoli-type
inequality on manifolds (see Lemma 7.1 in \cite{Li12}) to graphs. This
will imply the $L^q$ Liouville theorem for harmonic functions
when $q\in (1,\infty)$. Moreover, using this estimate, we shall prove
the main Theorem \ref{MT}.

\begin{thm}[Caccioppoli-type inequality]\label{Caccioppolli}
Let $f$ be a nonnegative subharmonic function on the weighted graph
$G.$ Then for any $1<q<\infty,$ $0<r<R-1,$ $r,R\in \N$
\begin{equation}\label{Caccio}
\sum_{e=xy\subset B_r}\mu_{xy}|\nabla_{xy}
f|^2\min\{f^{q-2}(x),f^{q-2}(y)\}\leq
\frac{C}{(R-r)^2}\sum_{B_{R}\setminus B_r}f^{q}(x)\mu_x.
\end{equation}
\end{thm}
\begin{rem}
The classical Caccioppoli inequality is the case $q=2.$ On graphs,
see e.g. \cite{CoulhonGrigoryan98,HolopainenSoardi97,LinXi10}.
\end{rem}
\begin{rem}
We take the convention that $0\cdot\infty=0$ on the LHS of
\eqref{Caccio} for the case $1<q<2$ because it suffices to consider
positive subharmonic functions by setting $f_{\epsilon}=f+\epsilon$
($\epsilon\to 0$).
\end{rem}
\begin{proof}
Fix a point $p\in G$ and denote the distance function to $p$ by
$r(x)=d(x,p).$ We denote by $B_r:=B_r(p)$ the closed ball centered
at $p$ of radius $r$. Let us choose a test function $\varphi$
satisfying
\[\varphi(x)=\left\{\begin{array}{ll}
1,& r(x)\leq r+1,\\
\frac{R-r(x)}{R-r-1},& r+1\leq r(x)\leq R,\\
  0, &R\leq r(x).
\end{array}\right.\]
Then $\supp \varphi\subset B_{R},$ $\nabla_{e}\varphi\neq 0$ only if
$e\subset B_R\setminus B_r,$ and $|\nabla_e\varphi|\leq
\frac{2}{R-r}$ for all $e\in E.$

For the fixed subharmonic function $f,$ we choose a particular
orientation of $E$ such that for any $e\in E,$ $e=xy$ where
$f(y)\geq f(x).$ We divide the proof into two cases.

\emph{Case 1. } $2\leq q <\infty$. Using $\varphi^2 f^{q-1}$ as
the test function, we have
\begin{eqnarray}\label{eq11}
0&\leq& \langle\Delta f, \varphi^2 f^{q-1}\rangle=-\langle\nabla f, \nabla (\varphi^2f^{q-1})\rangle\nonumber\\
&=& -\sum_{e\in E} \mu_e \nabla_e
f\nabla_e(\varphi^2f^{q-1})\nonumber\\
&=&-\sum_{e=xy\in E}\mu_e \nabla_e f\left[\varphi^2(x)\nabla_e
(f^{q-1})+\nabla_e\varphi(\varphi(x)+\varphi(y))f^{q-1}(y)
\right] \nonumber\\
&&\ (\mathrm{using}\ f^{q-1}(y)=\nabla_{xy}
(f^{q-1})+f^{q-1}(x))\nonumber\\
&=&-\sum_e \mu_e\nabla_e f\nabla_e
(f^{q-1})\varphi^2(y)-\sum_e\mu_e\nabla_e f\nabla_e
\varphi(\varphi(x)+\varphi(y))f^{q-1}(x)\nonumber\\
&&\ (\mathrm{by\ the\ mean\ value\ inequality\ }\nabla_e
(f^{q-1})\geq (q-1)\nabla_e
ff^{q-2}(x))\nonumber\\
&\leq&-(q-1)\sum_e \mu_e|\nabla_e f|^2
f^{q-2}(x)\varphi^2(y)-\sum_e\mu_e\nabla_e f\nabla_e
\varphi(\varphi(x)+\varphi(y))f^{q-1}(x)\nonumber\\
&&\\
&=& -(q-1)\sum_e\mu_e|\nabla_e f|^2
f^{q-2}(x)\varphi^2(y)-2\sum_e\mu_e\nabla_e
f\nabla_e\varphi\varphi(y)f^{q-1}(x)
\nonumber\\&&+\sum_e\mu_e\nabla_e
f|\nabla_e\varphi|^2f^{q-1}(x).\nonumber
\end{eqnarray}
Using Young's inequality for the second term in the last inequality,
$$2\nabla_e
f|\nabla_e\varphi|\varphi(y)f^{q-1}(x)\leq \frac{q-1}{2}|\nabla_e
f|^2 f^{q-2}(x)\varphi^2(y)+C |\nabla_e\varphi|^2f^{q}(x),
$$ we obtain that
\begin{eqnarray}\label{eq12}
0&\leq & -(q-1)/2\sum_e\mu_e|\nabla_e f|^2
f^{q-2}(x)\varphi^2(y)+C\sum_e\mu_e
|\nabla_e\varphi|^2f^{q}(x)\nonumber\\&&+\sum_e\mu_e\nabla_e
f|\nabla_e\varphi|^2f^{q-1}(x)\nonumber\\
&=&I+II+III.
\end{eqnarray}
By the choice of $\varphi$ we have
$$II\leq \frac{C}{(R-r)^2}\sum_{e\subset B_{R}\setminus B_r}\mu_e f^{q}(x)\leq \frac{C}{(R-r)^2}\sum_{x\in B_{R}\setminus B_r}\mu_x
f^{q}(x),$$ and
\begin{eqnarray*} III&\leq&\frac{C}{(R-r)^2}\sum_{e\subset B_{R}\setminus
B_r}\mu_e (f(y)-f(x))f^{q-1}(x)\leq \frac{C}{(R-r)^2}\sum_{e\subset
B_{R}\setminus B_r}\mu_e f(y)f^{q-1}(x)\\
&\leq & \frac{C}{(R-r)^2}\sum_{e\subset B_{R}\setminus B_r}\mu_e
f^{q}(y)\leq \frac{C}{(R-r)^2}\sum_{x\in B_{R}\setminus B_r} \mu_x
f^{q}(x).
\end{eqnarray*}
Hence by \eqref{eq12}, \begin{eqnarray*}\sum_{e=xy\subset
B_r}\mu_{e}|\nabla_e f|^2\min\{f^{q-2}(x),f^{q-2}(y)\}&\leq&
\sum_e\mu_e|\nabla_e f|^2 f^{q-2}(x)\varphi^2(y)\\ &\leq&
\frac{C}{(R-r)^2}\sum_{B_{R}\setminus
B_r}\mu_xf^{q}(x).\end{eqnarray*}

 \emph{Case 2.}  $1< q \leq 2.$ The only difference here is that
 by the mean value inequality
 $\nabla_{xy} (f^{q-1})\geq (q-1)\nabla_{xy} f f^{q-2}(y).$
By a similar calculation as in  Case 1, we have
\begin{eqnarray}\label{eq3}
0&\leq& \langle\Delta f, \varphi^2 f^{q-1}\rangle= -\sum_{e\in E}
\mu_e \nabla_e
f\nabla_e(\varphi^2f^{q-1})\nonumber\\
&=&-\sum_{e=xy\in E}\mu_e \nabla_e f\left[\nabla_e
(f^{q-1})\varphi^2(y)+f^{q-1}(x)\nabla_e\varphi(\varphi(x)+\varphi(y))
\right] \nonumber\\
&=&-\sum_e \mu_e\nabla_e f\left[\nabla_e
(f^{q-1})\varphi^2(x)+f^{q-1}(y)\nabla_e
\varphi(\varphi(x)+\varphi(y))\right]\nonumber\\
&\leq& -(q-1)\sum_e\mu_e|\nabla_e f|^2
f^{q-2}(y)\varphi^2(x)-2\sum_e\mu_e\nabla_e
f\nabla_e\varphi\varphi(x)f^{q-1}(y)
\nonumber\\&&-\sum_e\mu_e\nabla_e
f|\nabla_e\varphi|^2f^{q-1}(y)\nonumber\\
&\leq& -(q-1)\sum_e\mu_e|\nabla_e f|^2
f^{q-2}(y)\varphi^2(x)-2\sum_e\mu_e\nabla_e
f\nabla_e\varphi\varphi(x)f^{q-1}(y).
\end{eqnarray} Using Young's inequality for the second term and the
same argument as before, we
have
$$\sum_e\mu_e|\nabla_e f|^2 f^{q-2}(y)\varphi^2(x)\leq
\frac{C}{(R-r)^2}\sum_{B_{R}\setminus B_r}\mu_xf^{q}(x).$$ This
proves
$$\sum_{e=xy\subset B_r}\mu_{e}|\nabla_e
f|^2\min\{f^{q-2}(x),f^{q-2}(y)\}\leq
\frac{C}{(R-r)^2}\sum_{B_{R}\setminus B_r}\mu_xf^{q}(x)$$
which is \eqref{Caccio}.

\end{proof}

By this Caccioppoli-type inequality, we can prove the $L^q$
Liouville theorem for nonnegative subharmonic functions when $q\in
(1,\infty).$

\begin{coro}\label{coro1}
For any graph $G,$ there doesn't exist any nonconstant $L^q$
harmonic (nonnegative subharmonic) function for $q\in(1,\infty).$
\end{coro}
\begin{proof}
For any harmonic function $f,$ $|f|$ is subharmonic. Hence, it
suffices to prove the corollary for nonnegative subharmonic
functions. Suppose $f$ is an $L^q$ nonnegative subharmonic function
on $G.$ We apply Theorem \ref{Caccioppolli} by setting $R=2r.$ Since
the RHS of \eqref{Caccio} tends to zero as $r\to \infty,$ we have
for any $e=xy\in E$
\begin{equation}\label{eq1}|\nabla_e
f|\min\{f^{q-2}(x),f^{q-2}(y)\}=0.\end{equation}

Now we claim that $|\nabla_e f|=0$ for any $e\in E.$ Suppose not,
then there exists an $e\in E$ such that $|\nabla_e f|\neq0,$
w.l.o.g., we may assume $e=xy$ and $f(y)>f(x).$ By the equation
\eqref{eq1}, $\min\{f^{q-2}(x),f^{q-2}(y)\}=0.$ For $1<q<2,$
$\min\{f^{q-2}(x),f^{q-2}(y)\}>0$ which yields a contradiction. For
$2\leq q<\infty,$ we have $f(x)=0<f(y).$ For any $z\sim y$ and
$z\neq x,$ using the equation \eqref{eq1} for $yz,$ we have $f(z)=0$
or $f(z)=f(y).$ Noting that $f(x)<f(y),$ the subharmonicity of $f$ at
$y$ implies that
$$f(y)\leq \sum_w \frac{\mu_{yw}}{\mu_y}f(w)<f(y).$$ This is a
contradiction which proves the claim. Then $f$ is constant.
\end{proof}

Now we prove the main Theorem \ref{MT} which settles a question in
\cite{RigoliSalvatoriVignati97}.

\begin{proof}[Proof of Theorem \ref{MT}]
We divide the proof into two cases, $1<q\leq2$ and $2\leq q<\infty$.
In the following, we assume $\liminf_{R\to \infty}
\frac{1}{R^2}\sum_{B_R}f^{q}(x)\mu_x<\infty$ and show that $f$ is
constant. For any $r,R\in\N,$ $r+1<R,$ we define the test function
as
\[\varphi(x)=\varphi_{r,R}(x)=\left\{\begin{array}{ll}
1,& r(x)\leq r+1,\\
\frac{R-r(x)}{R-r-1},& r+1\leq r(x)\leq R,\\
  0, &R\leq r(x).
\end{array}\right.\]
Then $\nabla_{e}\varphi\neq 0$ only if $e\subset B_R\setminus B_r,$
and $|\nabla_e\varphi|\leq \frac{2}{R-r}$ for all $e\in E.$ In
addition, $\varphi(x)\leq 2\varphi(y)$ for any $e=xy\not\subset
B_R\setminus B_{R-2}.$

\emph{Case 1. } $1<q \leq 2$. Using the test function
$\varphi_{r,R}$ in \eqref{eq3}, we have
\begin{equation*}
\sum_{e\subset B_R}\mu_e|\nabla_e f|^2 f^{q-2}(y)\varphi^2(x)\leq
C(q)\sum_{e\subset B_R\setminus B_r}\mu_e\nabla_e
f|\nabla_e\varphi|\varphi(x)f^{q-1}(y).
\end{equation*}
Hence by the H\"older inequality,
\begin{eqnarray}\label{eq4}
&&\left(\sum_{e\subset B_R}\mu_e|\nabla_e f|^2
f^{q-2}(y)\varphi^2(x)\right)^2\nonumber\\&\leq&
C\left(\sum_{e\subset B_R\setminus B_r}\mu_e|\nabla_e f|^2
f^{q-2}(y)\varphi^2(x)\right)\left(\sum_{e\subset B_R\setminus
B_r}\mu_e|\nabla_e \varphi|^2 f^{q}(y)\right)\nonumber\\
&\leq& C\left((\sum_{e\subset B_R}-\sum_{e\subset B_r})\
\mu_e|\nabla_e f|^2
f^{q-2}(y)\varphi^2(x)\right)\left(\frac{C}{(R-r)^2}\sum_{x\in
B_R\setminus B_r}f^{q}(x)\mu_x\right).\nonumber\\
&&
\end{eqnarray}
Since we assume $\liminf_{R\to \infty}
\frac{1}{R^2}\sum_{B_R}f^{q}(x)\mu_x<\infty,$ there exists a
sequence $\{R_i\}_{i=1}^{\infty},$ $R_i\to \infty,$ such that
$R_{i+1}\geq 2R_i$ and
$\frac{1}{R_i^2}\sum_{B_{R_i}}f^{q}(x)\mu_x\leq K<\infty$ for all
$i\in\N.$ We define \begin{eqnarray}\label{eq21}
A_i&:=&\frac{1}{R_i^2}\sum_{B_{R_i}}f^{q}(x)\mu_x,\nonumber\\
\varphi_i&:=& \varphi_{R_i,R_{i+1}}(x),\\
Q_{i+1}&:=&\sum_{e\subset B_{R_{i+1}}}\mu_e|\nabla_e f|^2
f^{q-2}(y)\varphi_i^2(x).\nonumber
\end{eqnarray}
Now by setting $R=R_{i+1}$ and $r=R_i,$ the inequality \eqref{eq4}
reads as
\begin{eqnarray}\label{eq5}
Q_{i+1}^2&\leq&
C(Q_{i+1}-Q_i)\frac{R_{i+1}^2}{(R_{i+1}-R_i)^2}A_{i+1}\nonumber\\
&\leq& C(Q_{i+1}-Q_i) A_{i+1} \ \ (\ \mathrm{by}\ R_{i+1}\geq 2R_i).
\end{eqnarray}
Since $A_i\leq K$ for any $i\in \N,$
$$Q_{i+1}^2\leq CQ_{i+1}K.$$ This implies that $Q_{i+1}\leq CK$ for
all $i.$ On the other hand, \eqref{eq5} implies that
$$Q_{i+1}^2\leq CK(Q_{i+1}-Q_i).$$ Summing over $i$ in the above
inequality we have, for any integer $N$,
$$\sum_{i=1}^N Q_{i+1}^2\leq CK (Q_{N+1}-Q_1)\leq (CK)^2.$$
Hence $Q_i\to 0$ as $i\to\infty.$

\emph{Case 2. } $2\leq q <\infty$. This argument follows from the
idea of \cite{RigoliSalvatoriVignati97}. Using the test function
$\varphi_{r,R}$ in \eqref{eq11}, we have
\begin{eqnarray*}
&&\sum_{e\subset B_R}\mu_e|\nabla_e f|^2 f^{q-2}(x)\varphi^2(y)\leq
C(q)\sum_{e\subset B_R\setminus B_r}\mu_e\nabla_e
f|\nabla_e\varphi|(\varphi(x)+\varphi(y))f^{q-1}(x)\\
&& \ \ (\ \mathrm{by}\ \varphi(x)\leq 2\varphi(y)\ \mathrm{for}\
e=xy\not\subset B_R\setminus B_{R-2})\\
&\leq& 2C\sum_{e\subset B_R\setminus B_r}\mu_e\nabla_e
f|\nabla_e\varphi|\varphi(y)f^{q-1}(x)+C\sum_{e\subset B_R\setminus
B_{R-2}}\mu_e\nabla_e f|\nabla_e\varphi|^2f^{q-1}(x).
\end{eqnarray*} By the H\"older inequality, we have
\begin{eqnarray}\label{eq31}
&&\left(\sum_{e\subset B_R}\mu_e|\nabla_e f|^2
f^{q-2}(x)\varphi^2(y)\right)^2 \nonumber\\ &\leq&
C\left(\sum_{e\subset B_R\setminus B_r}\mu_e|\nabla_e f|^2
f^{q-2}(x)\varphi^2(y)\right)\left(\sum_{e\subset B_R\setminus
B_r}\mu_e f^{q}(x)|\nabla_e\varphi|^2\right)\nonumber\\
&&+ C\left(\sum_{e\subset B_R\setminus B_{R-2}}\mu_e|\nabla_e f|^2
f^{q-2}(x)|\nabla_e\varphi|^2\right)\left(\sum_{e\subset
B_R\setminus B_{R-2}}\mu_e f^{q}(x)|\nabla_e\varphi|^2\right)\nonumber\\
&\leq& C\left(\sum_{e\subset B_R\setminus B_r}\mu_e|\nabla_e f|^2
f^{q-2}(x)\varphi^2(y)+\frac{C}{(R-r)^2}\sum_{e\subset B_R\setminus
B_{R-2}}\mu_e |\nabla_e
f|^2f^{q-2}(x)\right)\nonumber\\
&&\times\left(\frac{C}{(R-r)^2}\sum_{x\in B_R\setminus B_r}\mu_x
f^{q}(x)\right)
\end{eqnarray}
We keep  $A_i$ and $\varphi_i$ as in
\eqref{eq21} and define two more quantities,
\begin{eqnarray*}
Q_{i+1}&:=&\sum_{e\subset B_{R_{i+1}}}\mu_e|\nabla_e f|^2
f^{q-2}(x)\varphi_i^2(y),\\
\beta_i&:=&\frac{C}{(R_{i+1}-R_{i})^2}\sum_{e\subset
B_{R_{i+1}}\setminus B_{{R_{i+1}}-2}}\mu_e |\nabla_e f|^2f^{q-2}(x).
\end{eqnarray*}
With these notations, the
inequality \eqref{eq31} reads as, by $R_{i+1}\geq 2R_i$,
\begin{equation}\label{eq41}
Q_{i+1}^2\leq CA_{i+1}(Q_{i+1}-Q_i+\beta_i)\leq
CK(Q_{i+1}-Q_i+\beta_i).
\end{equation} On the other hand, noting that $\nabla_e f\leq f(y),$
\begin{eqnarray*}\beta_i&\leq&\frac{C}{(R_{i+1}-R_{i})^2}\sum_{e\subset
B_{R_{i+1}}\setminus B_{{R_{i+1}}-2}}
\mu_ef^{q}(y)\\
&\leq& \frac{C}{(R_{i+1}-R_i)^2}\sum_{x\in
B_{R_{i+1}}\setminus B_{{R_{i+1}}-2}}f^{q}(x)\mu_x \ \ (\ \mathrm{by}\ R_{i+1}\geq 2R_i)\\
&\leq& CA_{i+1}\leq CK.
\end{eqnarray*}
Hence \eqref{eq41} implies that
$$Q_{i+1}^2\leq CK(Q_{i+1}+CK).$$ This implies that
that $Q_i$ is bounded, i.e. $Q_{i+1}\leq C(K).$ Hence $Q_i\leq Q_{i+1}\uparrow Q<\infty.$ By the definition of
$\beta_i$ and $Q_{i+1},$ we have
$$\beta_i\leq \frac{C}{(R_{i+1}-R_i)^2}Q\leq \frac{CQ}{R_i^2}\to 0,\ \ \ (i\to\infty).$$
Taking the limit $i\to \infty$ in \eqref{eq41}, we obtain $Q=0.$
This is the estimate we need for the Case 2.

In both cases, $Q_i\to 0$ as $i\to\infty.$ Since $\varphi_i(x)\equiv
1$ for $x\in B_{R_i},$ for any $R>0$ and sufficiently large $R_i\gg
R$ we have
$$\sum_{e\subset B_{R}}\mu_e|\nabla_e f|^2 \min\{f^{q-2}(x),f^{q-2}(y)\}\leq
Q_{i+1}\to 0, \ \ (i\to\infty).$$ Hence for each $e=xy\in E,$ we
have $|\nabla_e f|\min\{f^{q-2}(x),f^{q-2}(y)\}=0.$ A similar
argument as in Corollary \ref{coro1} shows that $f$ is constant.
\end{proof}

\begin{rem}
For the case $1<q\leq 2,$ we are now in the situation of
\eqref{eq4}. We may thus obtain the precise analogues of Theorem 2.2 in Karp
\cite{Karp82} and Theorem 1 ($b$) in Sturm \cite{Sturm94}. This is unknown for the case $q>2.$
\end{rem}

\section{Borderline case and counterexamples}
In this section, we deal with the borderline case, i.e. $q=1.$ We
adopt an idea of Li \cite{Li84,Li12} to prove the $L^1$ Liouville
theorem for nonnegative subharmonic functions. In our setting, we
don't need any curvature-like assumptions. We shall now complete the
proof of Theorem
\ref{Lqharm} by settling the $L^1$-case.

\begin{proof}[Proof of Theorem \ref{Lqharm}]
For the case $1<q<\infty,$ see Corollary \ref{coro1}. We only need
to prove the theorem for $q=1.$ Let $f$ be a nonnegative $L^1$
subharmonic function. We claim that $f$ is harmonic. The
subharmonicity of $f$ implies that $\Delta f=(P-I)f\geq 0,$ i.e. for
all $x\in G,$
\begin{equation}\label{eq2}Pf(x)\geq f(x),\end{equation} where $P$
is the transition operator. Since $f\in L^1(G)$ and $f\geq 0,$
\begin{eqnarray*}
\|Pf\|_1&=&\sum_xPf(x)\mu_x=\sum_{x,y}p(x,y)f(y)\mu_x\\
&=&\sum_{x,y}\mu_{xy}f(y)=\sum_{x,y}p(y,x)f(y)\mu_y\\
&=&\sum_y\left(\sum_xp(y,x)\right)f(y)\mu_y=\|f\|_1.
\end{eqnarray*} Hence by the monotonicity of \eqref{eq2}, we have $Pf(x)=f(x)$ for all
$x\in G.$ This proves the claim.

For any $a>0,$ we define a function $g:=\min\{f,a\}.$ Since $f$ is
harmonic, a straightforward calculation shows that $g$ is
superharmonic, i.e. $\Delta g=(P-I)g\leq 0.$ It is easy to see that
$0\leq g\in L^1(G)$ (by $g\leq f$). A similar computation as before
yields that $\|Pg\|_1=\|g\|_1.$ The monotonicity of $Pg\leq g$
implies that $g$ is harmonic, i.e. $Pg=g.$ Since $a$ is arbitrary,
we can prove that $f$ is constant. Suppose that $f$ is not constant,
then there exists $x,y\in G$ such that $x\sim y$ and $f(x)\neq
f(y).$ Without loss of generality, we may assume $f(x)<f(y).$ Now
choose $a=f(x).$ By the harmonicity of $f$ and $g,$
$$f(x)=g(x)=Pg(x)<Pf(x)=f(x).$$ A contradiction. This proves the
theorem.
\end{proof}

Now we give two examples to show that the $L^q$ Liouville theorem is
not true for $q\in(0,1).$ The first example is a graph of finite
volume and the second of infinite volume.
\begin{example}
Let $G=(V,E)$ be an infinite line, i.e. $V=\Z$ and $xy\in E$ iff
$|x-y|=1$ for $x,y\in \Z$. We define the edge weight as
$\mu_{xy}=2^{1-\max\{|x|,|y|\}}$ for $xy\in E.$ Obviously,
$\mu(G)<\infty.$ The function $f$ defined as
\[f(n)=\left\{\begin{array}{ll}
2^n-1,& n\geq 0,\\
1-2^{-n},& n<0,
\end{array}\right.
\] is a harmonic function on $G.$ Noting that $f(n)=-f(-n),$ we have
for any $q\in (0,1)$
$$\|f\|_q^q=2\sum_{n=1}^{\infty}f(n)^q(2^{-n+1}+2^{-n})\leq C\sum_{n=1}^{\infty}2^{(q-1)n}<\infty.$$
\end{example}
\begin{example}\label{expp1}
Let $\Gamma_1=(V_1,E_1)$ be any infinite graph with
$\mu(\Gamma_1)=\infty.$ Fix a vertex in $\Gamma_1,$ say $p\in V_1.$
Let $G=(V,E)$ be the weighted graph in the previous example and
$\underline{0}$ the vertex representing the origin of $\Z$. We
define a new graph $\Gamma=\Gamma_1\wedge G$ by gluing the vertices
$p$ and $\underline{0}$ (i.e. identifying $p$ with $\underline{0}$).
Formally, $\Gamma=(V_{\Gamma},E_{\Gamma})$ where
$V_{\Gamma}=(V_1\setminus\{p\})\cup V$ and $xy\in E_{\Gamma}$ iff
$xy\in E_1$ for $x,y\in V_1\setminus\{p\},$ or $xy\in E$ for $x,y\in
V,$ or $y=\underline{0}$ and $xp\in E_1,$ or $x=y=\underline{0}$ and
$pp\in E_1.$ Since no new edges is added, we take the edge weights
on $\Gamma$ to be those of $\Gamma_1$ and $G$. Then
\[g(x)=\left\{\begin{array}{ll}
0,& x\in V_1\setminus\{p\},\\
f(x),& x\in V,
\end{array}\right.
\] is a nonconstant $L^q$ harmonic functions on $\Gamma$ for $q\in(0,1).$
\end{example}
\begin{rem}\label{remq}
In both examples above, direct calculation shows that
$\sum_{B_R}|f(x)|\mu_x=O(R).$ This means that \eqref{meq1} fails for
$q=1.$
\end{rem}

Finally, as an application of Theorem
\ref{Lqharm}, we study $L^q$ ($1\leq q<\infty$) Liouville theorems
for solutions to higher order operators on graphs. Let
$\Delta^m:=\Delta\circ\Delta\circ\cdots\circ \Delta$ be the $m$-fold
composition of Laplace operators, i.e., the analogue of the polyharmonic operator in the continuous setting. A function $f\in
\R^V$ is therefore called \emph{polyharmonic} if $\Delta^m f=0$ for
some $m \ge 2.$ The maximum principle is not available
for polyharmonic functions. By our Theorem \ref{Lqharm}, we can
prove the $L^q$ ($1\leq q<\infty$) Liouville theorem for
polyharmonic functions on graphs of infinite volume.

\begin{thm}
Let $G$ be a graph of infinite volume, i.e. $\mu(G)=\infty$. If $f$
is an $L^q$ polyharmonic function on $G$ for $1\leq q<\infty,$ then
$f\equiv 0.$
\end{thm}
\begin{proof}
Suppose $\Delta^m f=0$ for some $m\in\N.$ We put $g^i:=\Delta^if$
for $1\leq i\leq m.$ Direct calculation (by Jensen's inequality)
shows that the normalized Laplace operator is a bounded operator on
$L^q$ for any $q\in [1,\infty).$ Hence, we know that $g^i\in L^q$ by
$f\in L^q.$ Since $\Delta g^{m-1}=0$ and $g^{m-1}\in L^q,$ by
Theorem \ref{Lqharm}, $g^{m-1}$ is constant. Since
$\mu(G)=\infty,$ $g^{m-1}\equiv 0.$ Iteratively, we get $f\equiv 0.$
\end{proof}

\bibliography{Lqharmonic}
\bibliographystyle{alpha}

\end{document}